\documentclass[12pt, reqno]{amsart}
\usepackage{amsmath, amsfonts, amsbsy, amsthm, amscd, graphicx}
\usepackage{amssymb, latexsym, color, comment}
\usepackage{bm}
\usepackage{enumerate}
\usepackage[svgnames,psnames]{xcolor}
\usepackage[colorlinks, citecolor=Green,linkcolor=FireBrick,linktocpage,unicode]{hyperref}

\theoremstyle{definition}
\numberwithin{equation}{section}
\newtheorem{theo}{Theorem}[section]

\newtheorem{lemm}[theo]{Lemma}
\newtheorem{rema}[theo]{Remark}

\begin{document}

\title{Second Variation Formula for the Laplace Eigenvalue Functional on Closed Manifolds}
\author{Kazumasa Narita}
\thanks{National Institute of Technology, Yonago College, Yonago, Tottori, 683-8502, Japan, narita@yonago-k.ac.jp, ORCID iD: 0009-0002-0833-5721}
\date{}

\maketitle

\begin{abstract} For a closed Riemannian manifold $(M,g)$ of dimension $n$, let $\lambda_{1}(g)$ be the first positive eigenvalue of the Laplace--Beltrami operator $\Delta_{g}$ and $\mbox{Vol}(M,g)$ the volume of $(M, g)$. Considering the scale-invariant quantity $\lambda_{k}(g)\mbox{Vol}(M,g)^{2/n}$  as a functional over all the metrics in a fixed conformal class, we derive a second variation formula for the functional. As a corollary, we prove that if the canonical flat metric on a torus is such that the multiplicity of $\lambda_{1}$ is two, then the flat metric is not a maximal point of the functional in its conformal class. This is a higher dimensional extension of Karpukhin's very recent work.
\end{abstract}



\section{Introduction}
Let $M$ be a (connected) closed manifold of dimension $n$. Given a Riemannian metric $g$ on $M$, the volume $\mbox{Vol}(M,g)$ and the Laplace--Beltrami operator $\Delta_{g}$ are defined. When there is no room for confusion, we sometimes abbreviate $\Delta_{g}$ as $\Delta$. Let $0 = \lambda_{0}(g) < \lambda_{1}(g) \leq  \lambda_{2}(g) \leq \cdots \leq \lambda_{k}(g) \leq \cdots$ be the eigenvalues of $\Delta_{g}$. The quantity $\overline{\lambda_{k}}(M, g) := \lambda_{k}(g)\mbox{Vol}(M,g)^{2/n}$ is invariant under scaling of the metric $g$. Set
\begin{equation*}
\Lambda_{k}(M, [g]) := \sup_{h \in [g]} \overline{\lambda_{k}}(M, h) \quad \mbox{and} \quad \Lambda_{k}(M) := \sup_{g} \overline{\lambda_{k}}(M, g),
\end{equation*}
where the supremum of the second quantity is taken over all the Riemannian metrics on $M$. When $M$ is fixed, we consider $\overline{\lambda_{k}}$ as a functional over all the Riemannian metrics on $M$.

When $M$ is a closed surface, $\Lambda_{1}(M)$ is bounded by a constant depending on its genus \cite{YY, Karpukhin-1}. Hersch \cite{Hersch} proved that on a $2$-dimensional sphere $S^{2}$, $\Lambda_{1}(S^{2})$ is realized precisely when the metric is round. For a 2-dimensional torus $T^{2}$, the set of its conformal classes is parametrized by the canonical flat metrics $g_{a,b}$ induced by the factorization of $\mathbf{R}^{2}$ by the lattice $\Gamma_{a,b} := \mathbf{Z}(1,0) \oplus \mathbf{Z}(a,b)$ $(0\leq a \leq 1/2, a^{2}+b^{2} \geq 1)$.  Nadirashvili \cite{Nadirashvili} proved that $\Lambda_{1}(T^{2})$ is realized precisely when the metric is the flat equilateral metric $g_{1/2, \sqrt{3}/2}$. El Soufi and Ilias \cite{EI-conformal} proved that each flat metric $g_{a,b}$ is a critical point of the functional $\overline{\lambda_{1}}\restriction_{[g_{a,b}]}$. In fact, El Soufi, Ilias and Ros \cite{EIR} proved that $\Lambda_{1}(T^{2}, [g_{a,b}])$ is realized precisely when $a^{2}+b^{2}=1$ and the metric is $g_{a,b}$. In particular, if $a^{2}+b^{2}>1$, then $g_{a,b}$ is a critical point of $\overline{\lambda_{1}}\restriction_{[g_{a,b}]}$, but not a maximum point. Even though it is not a maximum point, one might expect that it is a (local) maximal point of $\overline{\lambda_{1}}\restriction_{[g_{a,b}]}$. Against the expectation, Karpukhin \cite{Karpukhin} recently obtained a second variation formula for $\overline{\lambda_{1}}$ restricted to a conformal class on a closed surface, and proved the following:
\begin{theo}[\cite{Karpukhin}]
\label{Karpukhin's-Main-Theorem}
\textit{
 If $a^{2}+b^{2}>1$, then $g_{a,b}$ is not a maximal point of $\overline{\lambda_{1}}\restriction_{[g_{a,b}]}$.
 }
 \end{theo}

In this paper, we extend this Karpukhin's work to higher dimensions. On any closed manifold $M$ of dimension $n\geq 3$, there exists a one-parameter family $(g_{t})_{t>0}$ such that $\overline{\lambda_{1}}(M, g_{t})$ goes to infinity with $t$ \cite{CD}. Hence it is meaningless to consider $\Lambda_{k}(M)$. On the other hand, $\Lambda_{1}(M, [g])$ is always bounded by a constant \cite{EI-conformal-volume, LY}. With extra assumptions of $(M,[g])$, $\Lambda_{k}(M, [g])$ is also bounded by a constant $C(n,k)$ for any $k$ \cite{Korevaar, Hassannezhad}. In fact, El Soufi and Ilias \cite{EI-conformal-volume} proved that $\Lambda_{1}(S^{n}, [g_{\mathrm{can}}])$ is realized precisely by the round metric $g_{\mathrm{can}}$. This can be seen as a higher dimensional extension of aforementioned Hersch's work. Furthermore, they proved the following theorem in a separate paper \cite{EI-conformal}: 
\begin{theo}[\cite{EI-conformal}]
\label{EI-torus}
 Let $\Gamma$ be a lattice in $\mathbf{R}^{n}$ and $c(\Gamma^{*})$ the infimum of $\sum_{1\leq i \leq n} |w_{i}|^{2}$, where $\{w_{i}\}_{i=1}^{n}$ runs over all the basis of the dual lattice $\Gamma^{*}$. Set $T^{n}_{\Gamma} = \mathbf{R^{n}}/\Gamma$ and let $g_{\Gamma}$ be the canonical flat metric on $T^{n}_{\Gamma}$. Then one has
 \begin{equation*}
 \Lambda_{1}(T^{n}_{\Gamma}, [g_{\Gamma}]) \leq \frac{4\pi^{2}}{n}c(\Gamma^{*})
 \end{equation*}
 and the equality holds if and only if there exists a basis $\{w_{i}\}_{i=1}^{n}$ of $\Gamma^{*}$ such that $|w_{1}|^{2} = \cdots |w_{n}|^{2} = c(\Gamma^{*})/n$.
 \end{theo}
 This can be seen as a higher dimensional extension of the aforementioned work of El Soufi, Ilias and Ros \cite{EIR} and the equality condition in Theorem \ref{EI-torus} is precisely the condition $a^{2}+b^{2}=1$ in dimension two. In this paper, we extend Karpukhin's second variation formula to higher dimensions (Theorem \ref{Main-Theorem}) and prove the following:
\begin{theo}[Theorem \ref{Main-Theorem-Tori}]
\textit{
 Let $\Gamma$ be a lattice in $\mathbf{R}^{n}$. Set $T^{n}_{\Gamma} = \mathbf{R^{n}}/\Gamma$ and let $g_{\Gamma}$ be the canonical flat metric on $T^{n}_{\Gamma}$. Assume that there exists $w \in \Gamma^{*} \setminus \{ 0 \}$ such that the inequality $|w|<|v|$ holds for any $v \in \Gamma^{*} \setminus \{0, \pm w \}$. Then the canonical flat metric $g_{\Gamma}$ is not a maximal point of $\overline{\lambda_{1}}\restriction_{[g_{\Gamma}]}$.
}
\end{theo}
In dimension two, the assumption in the above theorem is precisely the condition $a^{2}+b^{2}>1$ and so the above theorem can be seen as a higher dimensional extension of Karpukhin's theorem \ref{Karpukhin's-Main-Theorem}.

\section{Main Results}
Let $(M,g)$ be a compact Riemannian manifold of dimension $n (\geq 2)$ and $d\mu_{g}$ its Riemannian measure. Set
\begin{equation*}
C^{\infty}_{0}(M,g) := \left\{ f \in C^{\infty}(M) \mid \int_{M} f d\mu_{g}=0  \right\}.
\end{equation*}
Fix a positive integer $k$. In this section, we always assume that $\lambda_{k-1}<\lambda_{k} = \lambda_{k+1} = \cdots =\lambda_{k+m-1} < \lambda_{k+m}$. Let $E_{k}=E_{k}(g)$ be the space of $\lambda_{k}$-eigenfunctions, that is, $E_{k} := \mbox{Ker}(\Delta-\lambda_{k} I)$. Let $\Pi_{k}$ be the $L^{2}$-orthogonal projection $L^{2}(M,g) \to E_{k}$. In this section, we always assume that $\varphi \in C^{\infty}_{0}(M,g)$ satisfies 
\begin{equation}
\label{condition}
  \int_{M} \varphi uv \: d\mu_{g} = 0 \quad \mbox{and}  \quad  \int_{M}u g( \nabla \varphi, \nabla v) \: d\mu_{g} = 0
\end{equation}
for any $u,v \in E_{k}$. Since the second equation in $(\ref{condition})$ can be rewritten as 
\begin{equation*}
\int_{M}\varphi \: \mbox{div}(u\nabla v) d\mu_{g} = 0,
\end{equation*}
and the both spaces $\{ u v \mid u, v \in E_{k}\}$ and $\{  \mbox{div}(u\nabla v) \mid u, v \in E_{k}\}$ are finite-dimensional, the existence of $\varphi \in C^{\infty}_{0}(M,g)$ satisfying $(\ref{condition})$ is guaranteed. In this paper, we always consider the deformation $(g_{t})_{t>0}$ given by $g_{t} := e^{\varphi t} g$. It is easy to see that $\Delta_{g_{t}}$ can be expressed as  
\begin{equation}
\label{Delta-gt}
\Delta_{g_{t}}f = e^{-\varphi t} \left\{ \Delta_{g}f -\frac{n-2}{2}t g(\nabla \varphi, \nabla f) \right\}, 
\end{equation}
where $f$ is an arbitrary smooth function on $M$. By the results of Bando and Urakawa \cite{BU}, there exists a real-analytic one-parameter family $\{ \Lambda_{j}(t), u_{j}(t) \}_{j=1}^{m}$ satisfying the following:
\begin{enumerate}
  \item $\{u_{j}(t)\}_{j=1}^{m}$ is orthonormal with respect to the $L^{2}(g_{t})$-inner product for any $t>0$.
  \item $\Lambda_{j}(0) = \lambda_{k}(g)$ for any $1\leq j \leq m$.
   \item $\Delta_{g_{t}} u_{j}(t) = \Lambda_{j}(t) u_{j}(t)$ for any $t>0$ and $1\leq j \leq m$.
\end{enumerate}
In particular, $\{u_{j}:= u_{j}(0)\}_{j=1}^{m}$ is an $L^{2}$-orthonormal basis of $E_{k}$.

In what follows, we use the above assumptions and notations to study the first and second variations of $\overline{\lambda_{k}}(g_{t}) := \lambda_{k}(g_{t}) \mbox{Vol}(M,g_{t})^{2/n}$. Our strategy is similar to that of Karpukhin's \cite{Karpukhin}, but our computation is much more complicated since the second term in $(\ref{Delta-gt})$ does not vanish in higher dimensions.
\begin{lemm}
\label{first-lemma}
\textit{
In the above setting, one has
\begin{equation*}
\left. \frac{d}{dt} \overline{\lambda_{k}}(g_{t}) \right|_{t=0} = 0.
\end{equation*}
}
\end{lemm}
\begin{proof}
The assumption that $\varphi \in C^{\infty}_{0}(M,g)$ implies
\begin{equation}
\label{derivative-of-volume}
\left. \ \frac{d}{dt}\mbox{Vol}(M,g_{t}) \right|_{t=0} = \left. \ \frac{d}{dt} \left(\int_{M} e^{n\varphi t/2} \: d\mu_{g} \right) \right|_{t=0} =  \frac{n}{2} \int_{M}\varphi \: d\mu_{g} = 0.
\end{equation}

Next we prove $\dot{\lambda}_{k} := \left. \ \frac{d}{dt}\lambda_{k}(g_{t})\right|_{t=0}= 0$. Let $P_{k, \varphi}$ be a linear endomorphism of $E_{k}$ defined by $P_{k, \varphi} := \Pi_{k} \circ \left. \frac{d}{dt}\Delta_{g_{t}}\right|_{t=0}$. $P_{k,\varphi}$ is symmetric with respect to the $L^{2}$-inner product. El Soufi and Ilias \cite{EI} showed that for each $1 \leq j \leq m$, $\dot{\Lambda}_{j} := \left. \ \frac{d}{dt}\Lambda_{j}(t)\right|_{t=0}$ is equal to one of the eigenvalues of $P_{k,\varphi}$. The formula $(\ref{Delta-gt})$ implies that for any $u, v \in E_{k}$, one has
\begin{equation*}
\begin{split}
&\quad \langle u, P_{k, \varphi}(v) \rangle_{L^{2}(M,g)} \\
&=\int_{M} u \left(\left. \frac{d}{dt}\Delta_{g_{t}}v \right|_{t=0}  \right) d\mu_{g} \\
&= \int_{M} \left[(-\varphi u \Delta_{g}v)-\frac{n-2}{2}ug(\nabla \varphi, \nabla v) \right]d\mu_{g} \\
&= -\lambda_{k}\int_{M}\varphi u v \: d\mu_{g} -\frac{n-2}{2}\int_{M}ug(\nabla \varphi, \nabla v) \: d\mu_{g} \\
&=0,
\end{split}
\end{equation*}
where the assumption $(\ref{condition})$ is used at the last equality. Hence we have shown that $P_{k,\varphi}$ is the zero mapping and so 
\begin{equation}
\label{Lambda}
\dot{\Lambda}_{j}  = 0
\end{equation}
for any $1 \leq j \leq m$. Since we have $\lambda_{k}(g_{t}) = \min\{\Lambda_{j}(t)\}_{j=1}^{m}$, one obtains $\dot{\lambda}_{k} = 0$. Thus combining this with $(\ref{derivative-of-volume})$, one concludes the assertion.
\end{proof}

In the following lemma, we find $\ddot{\lambda}_{k}$.

\begin{lemm}
\label{lemma-T} 
\textit{
Define $S_{k, \varphi}: E_{k}(g) \to L^{2}(M,g)$ by
\begin{equation*}
\begin{split}
&\quad S_{k, \varphi}(u) \\
&:= \lambda_{k}\varphi^{2}u-\sum_{i: \: \lambda_{i} \neq \lambda_{k}}\frac{2\lambda_{i}\lambda_{k}}{\lambda_{i}-\lambda_{k}}\varphi\Pi_{i}(\varphi u) +(n-2)\varphi g(\nabla \varphi, \nabla u) \\
&-(n-2) \sum_{i: \: \lambda_{i} \neq \lambda_{k}}\frac{1}{\lambda_{i}-\lambda_{k}} \left\{ \lambda_{i} \varphi \Pi_{i}\left( g\left(\nabla \varphi, \nabla u)\right) +\lambda_{k}g\left(\nabla \varphi, \nabla \left(\Pi_{i} (\varphi u \right) \right) \right) \right\} \\
&-\frac{(n-2)^{2}}{2} \sum_{i: \: \lambda_{i} \neq \lambda_{k}}\frac{1}{\lambda_{i}-\lambda_{k}} g\left(\nabla \varphi, \nabla \left( \Pi_{i}g\left( \nabla \varphi, \nabla u \right) \right) \right). \\
\end{split}
\end{equation*}
Then $T_{k, \varphi} := \Pi_{k} \circ S_{k, \varphi}$ is a symmetric linear endomorphism of $E_{k}$ with eigenvalues $\{ \ddot{\Lambda}_{j} \}_{j=1}^{m}$ and one has 
\begin{equation*}
\ddot{\lambda}_{k} =\min\{ \ddot{\Lambda}_{j} \}_{j=1}^{m}.
\end{equation*}
In particular, $\ddot{\lambda}_{k}$ is equal to the smallest eigenvalue of $T_{k, \varphi}$.
}
\end{lemm}

\begin{proof} 
Combining the formula $\Delta_{g_{t}}u_{j}(t) = \Lambda_{j}(t) u_{j}(t)$ with $(\ref{Delta-gt})$, one obtains
\begin{equation}
\label{beginning-eq}
e^{-\varphi t}\left\{\Delta_{g}u_{j}(t) -\frac{n-2}{2}tg(\nabla \varphi, \nabla u_{j}(t)) \right\} = \Lambda_{j}(t) u_{j}(t).
\end{equation}
Differentiating this by $t$, one obtains
\begin{equation*}
\begin{split}
&e^{-\varphi t} \left[ -\varphi  \Delta u_{j}(t) + \Delta u_{j}'(t) - \frac{n-2}{2} \left\{ (1-\varphi t)g\left( \nabla \varphi, \nabla u_{j}(t) \right) + t g\left( \nabla \varphi, \nabla u_{j}'(t) \right) \right\} \right] \\
&= \Lambda_{j}'(t)u_{j}(t) + \Lambda_{j}(t)u_{j}'(t). \\
\end{split}
\end{equation*}
Hence this equation and $(\ref{Lambda})$ imply that the derivative of $(\ref{beginning-eq})$ at $t=0$ is 
\begin{equation*}
-\lambda_{k}\varphi u_{j} +\Delta \dot{u}_{j} -\frac{n-2}{2} g\left( \nabla \varphi, \nabla u_{j} \right)  = \lambda_{k}\dot{u}_{j}.
\end{equation*}
Applying $\Pi_{i}$ to this equation, one obtains
\begin{equation*}
-\lambda_{k}\Pi_{i}(\varphi u_{j}) +\lambda_{i}\Pi_{i}( \dot{u}_{j}) -\frac{n-2}{2} \Pi_{i} \left( (g\left( \nabla \varphi, \nabla u_{j} \right)  \right) = \lambda_{k}\Pi_{i}(\dot{u}_{j}).
\end{equation*}
Hence one obtains
\begin{equation}
\label{dot-u}
\begin{split}
\dot{u}_{j} &= \Pi_{k}(\dot{u}_{j}) + \sum_{i: \: \lambda_{i} \neq \lambda_{k} } \Pi_{i}(\dot{u}_{j}) \\
&= \Pi_{k}(\dot{u}_{j}) + \sum_{i: \: \lambda_{i} \neq \lambda_{k} } \frac{1}{\lambda_{i}-\lambda_{k}}\left\{ \lambda_{k}\Pi_{i}(\varphi u_{j}) + \frac{n-2}{2}\Pi_{i}\left( g(\nabla \varphi, \nabla u_{j}) \right) \right\}. \\
\end{split}
\end{equation}
Hence one obtains
\begin{equation}
\label{Delta-dot}
\Delta \dot{u}_{j} = \lambda_{k} \Pi_{k}(\dot{u}_{j})+  \sum_{i: \: \lambda_{i} \neq \lambda_{k} } \frac{\lambda_{i}}{\lambda_{i}-\lambda_{k}}\left\{ \lambda_{k}\Pi_{i}(\varphi u_{j}) + \frac{n-2}{2}\Pi_{i}\left( g(\nabla \varphi, \nabla u_{j}) \right) \right\}.
\end{equation}
On the other hand, the second derivative of $(\ref{beginning-eq})$ at $t=0$ is 
\begin{equation*}
\lambda_{k}\varphi^{2}u_{j}- 2\varphi\Delta \dot{u}_{j} + \Delta \ddot{u}_{j} + (n-2)\left\{ \varphi g(\nabla \varphi, \nabla u_{j} ) -g(\nabla \varphi, \nabla \dot{u}_{j} ) \right\} = \ddot{\Lambda}_{j}u_{j} + \lambda_{k}\ddot{u}_{j},
\end{equation*}
where we have used $(\ref{Lambda})$. Substituting $(\ref{Delta-dot})$ into this equation, one obtains
\begin{equation*}
\begin{split}
&\lambda_{k}\varphi^{2}u_{j}- 2\lambda_{k}\varphi\Pi_{k}(\dot{u}_{j}) - \sum_{i: \lambda_{i} \neq \lambda_{k}}\frac{2\lambda_{i}}{\lambda_{i}-\lambda_{k}}\left\{\lambda_{k}\varphi \Pi_{i}(\varphi u_{j}) + \frac{n-2}{2}\varphi \Pi_{i}(g(\nabla \varphi, \nabla u_{j})) \right\} \\
& + \Delta \ddot{u}_{j} + (n-2)\left\{ \varphi g(\nabla \varphi, \nabla u_{j} ) -g(\nabla \varphi, \nabla \dot{u}_{j} ) \right\}  \\
&= \ddot{\Lambda}_{j}u_{j} + \lambda_{k}\ddot{u}_{j}. \\
\end{split}
\end{equation*}
Applying $\Pi_{k}$ to this equation and using the fact that $\Pi_{k}(\varphi \Pi_{k}(\dot{u}_{j}) )= 0$, which immediately follows from $(\ref{condition})$, one obtains
\begin{equation*}
\begin{split}
&\lambda_{k}\Pi_{k}(\varphi^{2}u_{j})  - \sum_{i: \lambda_{i} \neq \lambda_{k}}\frac{2\lambda_{i}}{\lambda_{i}-\lambda_{k}}\left\{\lambda_{k}\Pi_{k}\left( \varphi \Pi_{i}(\varphi u_{j}) \right) + \frac{n-2}{2}\Pi_{k}\left( \varphi  \Pi_{i}(g(\nabla \varphi, \nabla u_{j})) \right) \right\} \\
&+ (n-2) \Pi_{k} \left(  \varphi g(\nabla \varphi, \nabla u_{j} )   -g(\nabla \varphi, \nabla \dot{u}_{j} ) \right) \\
&= \ddot{\Lambda}_{j}u_{j} \\
\end{split}
\end{equation*}
Substituting $(\ref{dot-u})$ into this equation and using the fact that $\Pi_{k}\left(g\left(\nabla \varphi, \nabla(\Pi_{k}(\dot{u}))\right) \right)=0$, which follows from $(\ref{condition})$, one can obtain $T_{k, \varphi}(u_{j}) =  \ddot{\Lambda}_{j}u_{j}$. Since $\{u_{j}\}_{j=1}^{m}$ is an $L^{2}$-orthonormal basis of $E_{k}$, $T_{k, \varphi}$ is a symmetric linear endomorphism of $E_{k}$ with eigenvalues $\{ \ddot{\Lambda}_{j} \}_{j=1}^{m}$. On the other hand, by $(\ref{Lambda})$, one has the Taylor series 
\begin{equation*}
\Lambda_{j}(t) = \lambda_{k} + \frac{\ddot{\Lambda}_{j}}{2}t^{2} + O(t^{3}) \quad (t \to 0).
\end{equation*}
Since one has $\lambda_{k}(g_{t}) = \min\{ \Lambda_{j} (t) \}_{j=1}^{m}$, one obtains
\begin{equation*}
\ddot{\lambda}_{k} =\min\{ \ddot{\Lambda}_{j} \}_{j=1}^{m}.
\end{equation*}
\end{proof}
\begin{theo}
\label{Main-Theorem}
\textit{
Let $\mu_{k, \varphi}$ be the smallest eigenvalue of $T_{k, \varphi}$. Then  one has
\begin{equation*}
\overline{\lambda_{k}}(g_{t}) = \lambda_{k} + \frac{\alpha_{k, \varphi}}{2}t^{2} + o(t^{2}) \quad (t \to 0),
\end{equation*}
where $\alpha_{k, \varphi}=  \left( \mu_{k, \varphi}  \mathrm{Vol}(M,g) + \frac{n\lambda_{k}}{2}\| \varphi \|^{2}_{L^{2}(M, g)} \right) \left(\mathrm{Vol}(M,g) \right)^{(2-n)/n}$.
}
\end{theo}
\begin{proof}
We have
\begin{equation*}
\mbox{Vol}(M, g_{t}) = \int_{M}e^{n\varphi  t/2} d\mu_{g}
\end{equation*}
and so 
\begin{equation*}
\left. \ \frac{d^{2}}{dt^{2}}\mbox{Vol}(M,g_{t}) \right|_{t=0} = \frac{n^{2}}{4} \int_{M} \varphi^{2} d \mu_{g} = \frac{n^{2}}{4}\| \varphi \|^{2}_{L^{2}(M, g)}.
\end{equation*}
Combining this fact with $(\ref{derivative-of-volume})$, we obtain
\begin{equation*}
\left. \ \frac{d^{2}}{dt^{2}}\left(\mbox{Vol}(M,g_{t}) \right)^{2/n} \right|_{t=0} = \frac{n}{2} \left(\mbox{Vol}(M,g) \right)^{(2-n)/n}\| \varphi \|^{2}_{L^{2}(M, g)}.
\end{equation*}
Hence this fact and Lemma \ref{first-lemma} and \ref{lemma-T} imply
\begin{equation*}
\left. \frac{d^{2}}{dt^{2}} \overline{\lambda_{k}}(g_{t}) \right|_{t=0} = \mu_{k, \varphi} \left( \mbox{Vol}(M,g) \right)^{2/n} + \frac{n\lambda_{k}}{2} \left(\mbox{Vol}(M,g) \right)^{(2-n)/n}\| \varphi \|^{2}_{L^{2}(M, g)}
\end{equation*}
\end{proof}

\begin{rema} We explain the relationships between Karpukhin's result \cite{Karpukhin} and ours. Karpukhin studied the second variation formula for only closed surfaces. In order to extend his result to higher dimensional manifolds, we have added an extra assumption
\begin{equation*}
\int_{M} ug (\nabla \varphi, \nabla v) d\mu_{g} = 0 \quad \mbox{for any} \quad u, v \in E_{k}.
\end{equation*}
For closed surfaces $(n=2)$, the above $T_{k, \varphi}$ is symmetric with respect to the $L^{2}(g)$-inner product, but for higher dimensional manifolds, $T_{k, \varphi}$ is only diagonalizable (Lemma \ref{lemma-T}).  For closed surfaces, Karpukhin defined the quadratic form $Q_{k, \varphi}$ on $E_{k}$ by
\begin{equation*}
Q_{k, \varphi} = \frac{1}{\lambda_{k}} \langle u, T_{k, \varphi} u \rangle_{L^{2}(g)}
\end{equation*}
and expressed the Taylor series of $\overline{\lambda_{k}}(g_{t})$ using the smallest eigenvalue of $Q_{k, \varphi}$. Hence there appears to be a difference between Karpukhin's result and ours by $1/\lambda_{k}$, but they are essentially the same for $n=2$.
\end{rema}

Next we apply Theorem \ref{Main-Theorem} to certain flat tori. Let $\Gamma$ be a lattice in $\mathbf{R}^{n}$. Then $\Gamma$ can be expressed as $\Gamma = A \mathbf{Z}^{n}$ for some invertible matrix $A$ and the dual lattice $\Gamma^{*}$, which is defined by
\begin{equation*}
\Gamma^{*} = \{ w \in \mathbf{R}^{n} \mid w\cdot x \in \mathbf{Z} \quad \mbox{for any} \quad  x \in \Gamma \},
\end{equation*}
can be expressed as $\Gamma^{*} = (A^{T})^{-1} \mathbf{Z}^{n}$. Set $T^{n}_{\Gamma} := \mathbf{R}^{n}/\Gamma$ and let $g_{\Gamma}$ be the canonical flat metric on it. Then one has $\mbox{Vol}(T^{n}_{\Gamma}, g_{\Gamma}) = |\mbox{det}A|$. By the result of El Soufi and Ilias \cite{EI-conformal}, $g_{\Gamma}$ is a critical point of $\overline{\lambda_{1}}\restriction_{[g_{\Gamma}]}$.
\begin{theo}
\label{Main-Theorem-Tori}
\textit{
In the above setting, assume that there exists $w \in \Gamma^{*} \setminus \{ 0 \}$ such that the inequality $|w|<|v|$ holds for any $v \in \Gamma^{*} \setminus \{0, \pm w \}$. Then the canonical flat metric $g_{\Gamma}$ is not a maximal point of $\overline{\lambda_{1}}\restriction_{[g_{\Gamma}]}$.
}
\end{theo}

\begin{proof}
By the assumption, one has $\lambda_{1}(g_{\Gamma}) = 4\pi^{2}|w|^{2}$ with multiplicity two. More explicitly, 
\begin{equation*}
u_{1}(x) := \frac{\sqrt{2}}{|\mbox{det}A|^{1/2} } \sin (2\pi w \cdot x), \quad u_{2}(x) := \frac{\sqrt{2}}{|\mbox{det}A|^{1/2} } \cos (2\pi w \cdot x)
\end{equation*}
form an $L^{2}$-orthonormal basis of $E_{1}$. Set $\varphi= u_{1}$ and consider the conformal deformation $g_{t} := e^{\varphi t} g_{\Gamma}$. In order to use Theorem $\ref{Main-Theorem}$, we first verify that the assumption $(\ref{condition})$ is satisfied. Clearly, one has $ \varphi \in C^{\infty}_{0}(T^{n}_{\Gamma}, g_{\Gamma})$. Moreover, one has 
\begin{equation*}
\varphi u_{1} (x) = u_{1}^{2}(x)= \frac{2}{|\mbox{det}A| }\sin^{2} (2\pi w \cdot x) = \frac{1}{|\mbox{det}A| } \left( 1-\cos(4\pi w \cdot x) \right)
\end{equation*}
and 
\begin{equation*}
\varphi u_{2} (x) =u_{1}u_{2}(x)= \frac{2}{|\mbox{det}A| }\cos (2\pi w \cdot x) \sin (2\pi w \cdot x) = \frac{1}{|\mbox{det}A| } \sin(4\pi w \cdot x).
\end{equation*}
Both $\cos (4\pi w \cdot x)$ and $\sin(4\pi w \cdot x)$ are $4\lambda_{1}$-eigenfunctions. Hence $\varphi u_{1}$ and $\varphi u_{2}$ are orthogonal to $E_{1}$. Furthermore, one has
\begin{equation*}
g_{\Gamma} (\nabla \varphi, \nabla u_{1}) = \frac{2\lambda_{1}}{|\mbox{det}A|}\cos^{2}(2\pi w \cdot x) = \frac{\lambda_{1}}{|\mbox{det}A|}\left(1+\cos (4\pi w \cdot x) \right)
\end{equation*}
and 
\begin{equation*}
g_{\Gamma} (\nabla \varphi, \nabla u_{2}) = -\frac{2\lambda_{1}}{|\mbox{det}A|}\cos(2\pi w \cdot x) \sin (2\pi w \cdot x) = -\frac{\lambda_{1}}{|\mbox{det}A|}\sin (4\pi w \cdot x) .
\end{equation*}
Hence $g_{\Gamma} (\nabla \varphi, \nabla u_{1})$ and $g_{\Gamma} (\nabla \varphi, \nabla u_{2})$ are orthogonal to $E_{1}$ and thus $(\ref{condition})$ is satisfied. 

Next we compute the eigenvalues of $T_{1, \varphi}$, which is defined in Lemma \ref{lemma-T}. By the triple-angle identity, we have
\begin{equation*}
\begin{split}
u_{1}^{3} (x)&= \frac{2\sqrt{2}}{|\mbox{det}A|^{3/2}}\sin^{3}(2\pi w \cdot x)\\
&=  \frac{3\sqrt{2}}{2|\mbox{det}A|^{3/2}}\sin(2\pi w \cdot x) - \frac{\sqrt{2}}{2|\mbox{det}A|^{3/2}}\sin(6\pi w \cdot x) \\
&=  \frac{3}{2|\mbox{det}A|}u_{1}(x) - \frac{\sqrt{2}}{2|\mbox{det}A|^{3/2}}\sin(6\pi w \cdot x) \\
\end{split}
\end{equation*}
Furthermore, by the product-to-sum formulas, we obtain
\begin{equation*}
\begin{split}
u_{1}(x) \cos(4\pi w\cdot x) &=  \frac{\sqrt{2}}{|\mbox{det}A|^{1/2}}\sin(2\pi w \cdot x) \cos(4\pi w \cdot x) \\
&=  \frac{\sqrt{2}}{2|\mbox{det}A|^{1/2}}\left\{ \sin(6\pi w \cdot x)-\sin (2\pi w \cdot x) \right\}  \\
&= -\frac{1}{2}u_{1}(x) +  \frac{\sqrt{2}}{2|\mbox{det}A|^{1/2}}\sin(6\pi w \cdot x) \\
\end{split}
\end{equation*}
and 
\begin{equation*}
\begin{split}
g_{\Gamma}\left(\nabla u_{1}(x), \nabla(\cos(4\pi w \cdot x) \right) &= -\frac{8\sqrt{2}|w|^{2} \pi^{2} }{|\mbox{det}A|^{1/2}} \sin(4\pi w \cdot x) \cos(2\pi w \cdot x)\\
&= -\frac{\sqrt{2}\lambda_{1} }{|\mbox{det}A|^{1/2}} \left\{ \sin(6\pi w \cdot x)+ \sin(2\pi w \cdot x) \right\} \\
&= -\lambda_{1}u_{1}(x) -\frac{\sqrt{2}\lambda_{1} }{|\mbox{det}A|^{1/2}}  \sin(6\pi w \cdot x). \\
\end{split}
\end{equation*}
Together with these formulas, a straightforward computation yields
\begin{equation*}
\begin{split}
S_{1, \varphi}(u_{1}) (x)  &= \frac{\lambda_{1}}{6|\mbox{det}A|}\left\{1+5(n-2)+(n-2)^{2} \right\}u_{1}(x) \\
&\quad + \frac{\sqrt{2}\lambda_{1}}{6|\mbox{det}A|^{3/2}}\left\{5-3(n-2)+(n-2)^{2} \right\}\sin(6\pi w \cdot x).\\
\end{split}
\end{equation*}
Since the function $\sin(6\pi w \cdot x)$ is a $9\lambda_{1}$-eigenfunction, we obtain
\begin{equation*}
T_{1, \varphi}(u_{1}) = \frac{\lambda_{1}}{6|\mbox{det}A|}(n^{2}+n-5 )u_{1}.
\end{equation*}
In a similar manner, one can obtain
\begin{equation*}
u_{1}^{2}u_{2}(x) = \frac{1}{2 |\mbox{det}A|}u_{2}(x)- \frac{\sqrt{2}}{2 |\mbox{det}A|^{3/2}}\cos(6\pi w \cdot x),
\end{equation*}
\begin{equation*}
u_{1}(x)\sin(4\pi w \cdot x) = \frac{1}{2}u_{2}(x) -\frac{\sqrt{2}}{2 |\mbox{det}A|^{1/2}}\cos(6\pi w \cdot x),
\end{equation*}
and
\begin{equation*}
g_{\Gamma}\left(\nabla u_{1}(x), \nabla(\sin(4\pi w \cdot x) \right) = \lambda_{1}u_{2}(x) +\frac{\sqrt{2}}{2 |\mbox{det}A|^{1/2}}\cos(6\pi w \cdot x).
\end{equation*}
These equations imply
\begin{equation*}
\begin{split}
S_{1, \varphi}(u_{2})(x)  &= \frac{\lambda_{1}}{6|\mbox{det}A|}\left\{-5-(n-2)+(n-2)^{2} \right\}u_{2}(x) \\
&\quad + \frac{\sqrt{2}\lambda_{1}}{6|\mbox{det}A|^{3/2}}\left\{5-3(n-2)+(n-2)^{2} \right\}\cos(6\pi w \cdot x).\\
\end{split}
\end{equation*}
and hence
\begin{equation*}
T_{1, \varphi}(u_{2}) = \frac{\lambda_{1}}{6|\mbox{det}A|}(n^{2}-5n+1 )u_{2}.
\end{equation*}
Hence one obtains
\begin{equation*}
\mu_{1, \varphi} =  \frac{\lambda_{1}}{6|\mbox{det}A|}(n^{2}-5n+1 ).
\end{equation*}
Thus we conclude that $\alpha_{1, \varphi}$, which is defined in Theorem $\ref{Main-Theorem}$, is given by
\begin{equation*}
\alpha_{1, \varphi} = \frac{(n-1)^{2}\lambda_{1}}{6}|\mathrm{det}A|^{(2-n)/n} >0.
\end{equation*}
Thus the canonical flat metric $g_{\Gamma}$ is not $\overline{\lambda_{1}}$-conformally maximal.
\end{proof}
As we have remarked in Introduction, this result is a higher dimensional extension of Theorem \ref{Karpukhin's-Main-Theorem} by Karpukhin \cite{Karpukhin}.

\vspace{0.4in}

\textbf{Acknowledgements.} I would like to thank Professor Shin Nayatani for his constant encouragement and valuable comments on this work.

\end{document}